\documentclass[a4paper,reqno,12pt]{amsart}

\usepackage[a4paper,width={16cm},left=2.5cm,bottom=3cm, top=3cm, marginparwidth=50pt]{geometry}
\usepackage{enumerate}

\usepackage[english]{babel}

\usepackage[dvipsnames]{xcolor}
\usepackage{amsmath}
\usepackage{amssymb}
\usepackage{amsfonts}
\usepackage{amsthm,graphicx}
\usepackage{resmes}

\usepackage{resmes}
\usepackage{hyperref}

\usepackage{mathtools}
\mathtoolsset{showonlyrefs}
\numberwithin{equation}{section}
\newtheorem{theorem}{Theorem}[section]

\newtheorem{corollary}[theorem]{Corollary}

\newtheorem{proposition}[theorem]{Proposition}
\theoremstyle{definition}  
\newtheorem{definition}[theorem]{Definition}

\newtheorem{remark}[theorem]{Remark}

\newcommand{\mb}{\mathbb}

\newcommand{\la}{\lambda}
\newcommand{\norm}[1]{\left\lVert#1\right\rVert}

\newcommand{\R}{\mb{R}}

\newcommand{\e}{\varepsilon}

\def\vint_#1{\mathchoice
	{\mathop{\vrule width 6pt height 3 pt depth -2.5pt
			\kern -8pt \intop}\nolimits_{#1}}%
	{\mathop{\vrule width 5pt height 3 pt depth -2.6pt
			\kern -6pt \intop}\nolimits_{#1}}%
	{\mathop{\vrule width 5pt height 3 pt depth -2.6pt
			\kern -6pt \intop}\nolimits_{#1}}%
	{\mathop{\vrule width 5pt height 3 pt depth -2.6pt
			\kern -6pt \intop}\nolimits_{#1}}}

\usepackage{subfig}
\usepackage{tikz}
\usepackage{xcolor}
\usepackage{pgfplots}
\pgfplotsset{compat=1.18}
\usepackage{mathrsfs}

\title[A regularity theorem for stationary measures]
{A regularity theorem for stationary measures}

\author[L. Ambrosio  and G. Siclari]{Luigi Ambrosio and Giovanni Siclari}

\address{Giovanni Siclari 
\newline \indent Centro di Ricerca Matematica Ennio De Giorgi
\newline \indent Scuola Normale Superiore di Pisa
\newline\indent Piazza dei Cavalieri 3, 56126 Pisa, Italy}
\email{giovanni.siclari@sns.it}

\address{Luigi Ambrosio 
\newline \indent Scuola Normale Superiore di Pisa
\newline\indent Piazza dei Cavalieri 3, 56126 Pisa, Italy}
\email{luigi.ambrosio@sns.it}

\date{\today}

\begin{document}

\begin{abstract}
We investigate a variational  problem for eigenvalues of the Laplace-Beltrami operator on  smooth manifolds with respect to Radon measures belonging to a suitable class; we are motivated by conformal eigenvalues in dimension two.
Our main result is a regularity result for stationary measures with respect to outer variations. More precisely, we prove that any sufficiently regular stationary measure is absolutely continuous with respect to the classical volume measure and that its density is induced by an harmonic map. Our result has some interesting applications to Steklov eigenvalues on subdomains.
\end{abstract}

\maketitle


\medskip 


%53C43: Differential geometric aspects of harmonic maps;
%49K35: Optimality conditions for minimax problems;
%47A75: Eigenvalue problems for linear operators;
%49R05:Variational methods for eigenvalues of operators;

\maketitle

\section{Introduction}\label{sec_intr}
Let $(M,g)$ be a $d$-dimensional smooth closed (i.e. compact and without boundary) Riemannian manifold, let $-\Delta_g$ the Laplace Beltrami operator and let $\nu$ be a Radon measure on $M$ such that
\begin{equation}\label{hp_compct}
 H^1(M,g) \hookrightarrow L^2(M,\nu) \text{ is compact}
\end{equation}
meaning more precisely that the embedding of $C^\infty(M)$ into $L^2(M,\nu)$ has a linear, compact extension to $H^1(M,g)$).
Under this assumption, by classical spectral theory, the spectrum of the eigenvalue problem 
\begin{equation}\label{eq_eigen_standard_empty}
-\Delta_g u=\la u\nu  \quad \text{ in }  M,
\end{equation}
is given by a diverging sequence $\{\la_n(g,\nu)\}_{n \in \mathbb{N}}$ of positive eigenvalues repeated according to their multiplicity, see Section~\ref{subsec_fun_sec} for further details.

\begin{definition}\label{def_sta}
For any measures $\mu, \nu$ satisfying \eqref{hp_compct}, we say that 
\begin{equation}\label{def_kue}
\mu_{\nu}(\e):=\mu+\e(\nu-\mu), \quad \text{ for any } \e \in [0,1]
\end{equation}
is the  outer variation of $\mu$  with respect to $\nu$.
We say that a measure $\mu$   is   \textit{stationary from above with respect to outer variations} or simply \textit{stationary} for some  $n \in \mathbb{N}\setminus\{0\}$ if
\begin{equation}
\limsup_{\e \to 0^+}\frac{1}{\e}[\la_{n}(g,\mu_\nu(\e))-\la_{n}(g,\mu)] \le 0,
\end{equation}
for any outer variation.
\end{definition}
In this paper, we investigate the regularity of stationary measures. We are motivate by the problem 
\begin{equation}\label{prob_max_mes_comp}
\widetilde{\Lambda}_n(M,g):=\sup\{\la_n(g,\nu):\nu \text{ Radon}, \nu(M)=1, H^1(M,g) \hookrightarrow L^2(M,\nu) \text{ is compact}\},
\end{equation}
whose maximizers are clearly stationary by definition. Problem  \eqref{prob_max_mes_comp} has  strong links, in dimension $2$, to the maximization of Laplace–Beltrami eigenvalues over a conformal class of metrics, a topic commonly referred to as the conformal spectrum (see \cite{CS_conf,FN_confrm_first}) as we detail in the next section.

\subsection{Connections with the conformal spectrum in dimension $2$}
Let $(M,g)$ be  a closed  Riemann surface. Denoting by $[g]$ the conformal class of  $g$, that is, the equivalence class of all the metrics that can be obtained with a conformal change of metric starting from $g$,  the conformal spectrum is defined as  
\begin{equation}\label{prob_max_conform}
\Lambda_n(M,[g]):=\sup\{\la_n(g){\rm Vol}_g(M): g \in [g]\}=\sup\{\la_n(g):g \in [g], {\rm Vol}_g(M)=1\}.
\end{equation}
It is clear that, to avoid a trivial problem, some normalization on ${\rm Vol}_g$ is needed since the Rayleigh quotient transform as 
\begin{equation}\label{eq_conformal_change}
\frac{\int_M |\nabla_{g_\phi} v|^2 \,  {\rm dVol}_{g_\phi}}{\int_M | v|^2 \, {\rm dVol}_{g_\phi}}
=\frac{\int_M |\nabla_g v|^2 \,  {\rm dVol}_g}{\int_M | v|^2 e^{\phi}\, {\rm dVol}_g},
\end{equation}
under a conformal change of metric $g_\phi:= e^{\phi} g$.

There is a vast literature concerning conformal eigenvalues for Riemannian surfaces. First of all, the exact value  of $\Lambda_n(M,[g])$ for different surfaces with the standard metric or global estimates on $\Lambda_n(M,[g])$ were  the object of study of many papers, see  \cite{CS_conf,EI_imm,H_class,K_var,K_eigen_bounded,YY_eigen_bounds}. In particular it is known that 
\begin{equation}\label{ineq_eigen_M}
\la_n(g){\rm Vol}_g(M) \le C(M) n,
\end{equation}
for any metric $g$, some constant $C(M)$ depending only on $M$  and any $n \in \mathbb{N}$, thus $\Lambda_n(M,[g])$ is finite.

Furthermore, the existence and regularity of maximizers has also been deeply investigated.
It is an interesting and challenging problem due to the lack of any reasonable compactness properties in the class of conformal metrics and the presence of concentration phenomena which may prevent existence as it happens  for $\Lambda_2(\mb{S}^2,[g])$, see \cite{N_second_sphere}.
It  has been addressed in several papers, we refer for example to \cite{KNPP_exi,NS_exi_first,NS_exi_higher,P_exi_first,P_exi_high}. The known results can be roughly summarized as follow. If 
$\Lambda_{n+1}(M,[g]) > \Lambda_{n}(M,[g])+8 \pi$ then there exists a metric $g$, which may admit a finite number of conical singularities, that maximizes \eqref{prob_max_conform}. In particular, $g$ is smooth but for a finite number of points where it can degenerate to $0$. The constant $8 \pi$ is actually equal to  $\Lambda_1(\mb{S}^2,[g])$ and it's relevance is due to the fact that it is  the maximal energy level of a sequence of metrics concentrating to a  point, see \cite{K_var}.

A possible strategy to deal with  the lack of compactness,  
 thanks to  \eqref{eq_conformal_change},  is   to consider the relaxed problem 
\begin{equation}\label{max_conform_meas}
\sup\{\la_n(g,\nu)\nu(M)\}=\sup\{\la_n(g,\nu): \nu(M)=1\},
\end{equation}
where for any Radon measure $\mu$ on $M$ we let
\begin{equation}\label{def_eigen_min_max}
\la_n(g,\nu):=\inf_{E_n}\sup_{v \in E_n} \frac{\int_M |\nabla_g v|^2 \,  {\rm dVol}_g }{\int_M | v|^2 \, d \nu},
\end{equation}
and $E_n$ is any $(n+1)$-dimensional subspace of $C^\infty(M)$ that remains $(n+1)$-dimensional in $L^2(M,\nu)$.
This approach was introduced by  \cite{K_var} and further investigated  in \cite{NS_exi_higher}. 

The variational eigenvalues defined in \eqref{def_eigen_min_max} for a generic Radon measure can present several pathologies. For example, if $\nu$ is a sum of $m$ Dirac's delta, then $\la_n(g,\nu)$  is either $0$ or $+\infty$ depending on the relationship between $n$ and $m$. Furthermore, they are not the eigenvalue of Laplacian-Beltrami operator in $L^2(M,\nu)$ in a spectral sense, since $H^1(M,g) \not \hookrightarrow L^2(M,\nu)$, in general.

In conclusion, we can think about problem \eqref{prob_max_mes_comp} for $d=2$ as a relaxation of \eqref{prob_max_conform} in the smaller class of Radon measures satisfying \eqref{hp_compct} (see also Proposition \ref{prop_reg} below).
Of course, if $d>2$,  this is no longer the case, as the Laplacian is no longer invariant under conformal transformations, but it still is an interesting variational problem.

\subsection{Main results}
Our first result is that it is enough to consider absolutely continuous measures with a positive smooth  density  with respect  ${\rm Vol}_g$ to reach the supremum in \eqref{prob_max_mes_comp}.
More precisely, letting 
\begin{equation}\label{prob_max_mes_reg}
\bar{\Lambda}_n(M,g):=\sup\left\{\la_n(g,f\, {\rm Vol}_g ):f \in C^\infty(M), f>0, \int_M f \, {\rm dVol}_g=1\right\},
\end{equation}
we have the following proposition.

\begin{proposition}\label{prop_reg}
Let $n \in \mathbb{N}$, $d \ge 2$ and let $(M,g)$ be  a closed, smooth  $d$-dimensional  manifold. Then 
\begin{equation}
\widetilde{\Lambda}_n(M,g)=\bar{\Lambda}_n(M,g)
\end{equation}
so that,  if $d=2$,
\begin{equation}
\Lambda_n(M,[g])=\widetilde{\Lambda}_n(M,g).
\end{equation}
\end{proposition}
Our approach is based on regularization on a measure $\nu$ by the means of the heat semigroup of $M$. Such an idea was already present in
\cite[Section 3.2]{P_exi_first}, where such a regularization was used to construct maximizing sequence converging to a maximum for the first conformal eigenvalue $\Lambda_1(M,[g])$.

Our main result is the following regularity theorem for stationary measures. 
In the case of maximizers with $d=2$, if the measure is absolutely continuous with respect to the volume measure,   such a result was already known,  see \cite{IS_lap,FS_min,N_be}. 
Furthermore, it was recently obtained for maximizers of \eqref{prob_max_mes_comp} in the case $n=1,2$ and $d=2$ in   \cite[Theorem 1.4, Theorem 1.14]{KS_min_max} with an approach very different from ours.

\begin{theorem}\label{theor_main}
Let $d \ge 2$ and let $(M,g)$ be  a closed, smooth  $d$-dimensional  manifold.
Let $\mu$ be a Radon probability measure on $M$ satisfying \eqref{hp_compct} 
and suppose that $\mu$ is stationary for some $n \in \mathbb{N}\setminus\{0\}$.
Then  the eigenvalue $\la_n(g,\mu)$ has multiplicity $k_n \ge 2$ and there exist $m_n \in\{2,\ldots,k_n\}$, positive numbers
$\{\gamma_1, \ldots, \gamma_{m_n}\}$ 
and an orthonormal basis $\{\varphi_{n,i}\}_{i=1,\ldots,k_n}$ of  eigenfunctions in $L^2(M,\mu)$ of the  associated eigenspace 
such that
\begin{equation}\label{eq_mu_volg}
\mu=\frac{\sum_{i=1}^{m_n}\gamma_i|\nabla_g \varphi_{n,i}|^2}{\lambda_n(g,\mu)} \, {\rm Vol}_g.
\end{equation}
Furthermore, the map
\begin{equation}
\Phi_n: M \to \mb{S}^{m_n-1}, \quad \Phi_n=(\sqrt{\gamma_1}\varphi_{n,i},\dots, \sqrt{\gamma_{m_n}}\varphi_{n,m_n}) 
\end{equation}
is  harmonic onto $\mb{S}^{m_n-1}$, that is,
\begin{equation}
|\Phi_n|=1 \quad \text{ and }\quad -\Delta_g \Phi_n=|\nabla_g \Phi_n|^2 \Phi_n.
\end{equation}
In particular, if $d=2$,   $\Phi_n \in C^{\infty}(M,\R^{m_n})$ so that the associated metric 
\begin{equation}
g_\mu:=\frac{\sum_{i=1}^{m_n}\gamma_i|\nabla_g \varphi_{n,i}|^2}{\lambda_n(g,\mu)} g
\end{equation}
is  smooth outside the finite set $\{|\nabla_g \Phi_n|=0\}$.
\end{theorem}

We remark that our proof of  Theorem \ref{theor_main} is rather short, making use of recent abstract results about quantitative spectral stability for compact operators, see \cite[Theorem 3.12]{BORS_abs_mul}. Indeed, we can apply such results to outer variations of $\mu$, to  
deduce from the stationarity the condition 
\begin{equation}\label{ineq_nu_mu_intro}
 \sup_{\varphi \in A_n} \int_M |\varphi|^2 d \nu \ge 1
\end{equation}
for any  measure $\nu$ satisfying \eqref{hp_compct}, where $A_n:=\left\{u \in E(\lambda_n(g,\mu)): \int_M u^2 \, d\mu=1\right\}$
and $E(\lambda_n(g,\mu))$ is the eigenspace associated to $\lambda_n(g,\mu)$.

From this condition, we will be able to prove Theorem \ref{theor_main}. It's important to notice how outer variations  allows us modify the support of $\mu$. This is a strong difference with the variations introduced in \cite{K_var} which instead kept the support fixed. This is the reason why we can show that \eqref{eq_mu_volg} holds in $M$ and not just  in the interior of the support of $\mu$,
as in  \cite[Theorem $C_k$]{K_var}.

As observed in   \cite{KS_min_max}, a result as  Theorem \ref{theor_main} has applications concerning the Steklov eigenvalues  on Lipschitz  subdomains $\Omega$ of $M$. We recall that the Steklov eigenvalues of Laplace-Beltrami operator on $\Omega$ 
\begin{equation}
\begin{cases}
-\Delta_g u=0, \quad \text{ in }  \Omega,\\
-\partial_\nu u=\sigma u, \quad \text{ on }  \partial \Omega,
\end{cases}
\end{equation}
are  a positive and diverging sequence $\{\sigma_n(\Omega,g)\}$ (repeating each eigenvalue according to its multiplicity).
By the Rayleigh quotient characterization,  letting $\mu_{\partial\Omega}$ be the surface measure of $\partial \Omega$, 
\begin{equation}
\sigma_n(\Omega,g)=\inf_{E_n}\sup_{v \in E_n} \frac{\int_\Omega |\nabla_g v|^2 \,  {\rm dVol}_g }{\int_M | v|^2 \, d\mu_{\partial\Omega}}
\le \inf_{E_n}\sup_{v \in E_n} \frac{\int_M |\nabla_g v|^2 \,  {\rm dVol}_g }{\int_M | v|^2 \, d\mu_{\partial\Omega}}=\la_n(g,\mu_{\partial\Omega}).
\end{equation}
Hence we have the following corollary which, for $d=2$, generalizes \cite[Theorem 1.6-Theorem 1.16]{KS_min_max}, where the cases $n=1$ and $n=2$ were dealt with.
Indeed, since the embedding of $H^1(M,g)$ in $L^2(M,\mu_{\partial\Omega})$ is compact  and $\mu_{\partial\Omega}$ is not absolutely continuous with respect to ${\rm Vol}_g$, by Theorem \ref{theor_main} it turns out that $\mu_{\partial\Omega}$ cannot be a maximizer of \eqref{prob_max_mes_comp}.

\begin{corollary}
Let $d \ge 2$ and let $(M,g)$ be  a closed, smooth  $d$-dimensional  manifold. Then for any Lipschitz  subdomain $\Omega \subset M$
\begin{equation}
    \sigma_n(\Omega,g) < \widetilde{\Lambda}_n(M,g).
\end{equation}
In particular, if $d=2$, 
\begin{equation}
    \sigma_n(\Omega,g) < \Lambda_n(M,[g]).
\end{equation}
\end{corollary}

\begin{remark}\label{remark_boundary}
We have focused on the case of closed manifolds for the sake of clarity of exposition. Theorem \ref{theor_main} can actually be proved in the same way for compact manifolds with boundary under Neumann boundary conditions with no modifications. A similar result for $d$-dimensional manifolds with boundary for smooth metrics was obtained in \cite{E_bound}. Actually our theorem generalizes  the  results in  \cite{E_bound} only in the case $d=2$, since in that paper the problem under consideration is always the problem of conformal eigenvalues in any dimension, which differs from \eqref{prob_max_mes_comp} if $d \neq 2$.
\end{remark}

The paper is organized as follows. In Section \ref{sec_func} we collect several preliminary facts about the functional setting, we  obtain a quantitative spectral stability result for outer variations and   we prove Proposition \ref{prop_reg} by a regularization argument. In Section \ref{sec_reg} we prove Theorem \ref{theor_main}.

\section{Functional setting, spectral stability,  outer variations and regularization}\label{sec_func}

\subsection{Functional setting}\label{subsec_fun_sec}

Let $\nu$ be a Radon probability measure $\nu$ on $M$ and assume that the Poincaré-type inequality
\begin{equation}\label{ineq_poinc}
\int_M|v-c_{v,\nu}|^2 d \nu \leq K_\nu \int_M |\nabla_g v|^2 {\rm dVol}_g\qquad \forall v\in C^\infty(M)
\end{equation}
holds, where
\begin{equation}
c_{v,\nu}:=\int_M v \, d\nu.
\end{equation}
Let us define $H_\nu^1(M,g)$ as the closure of $C^\infty(M)$ with respect to the norm 
\begin{equation}\label{def_sobolev_space_meausure}
\norm{v}^2_{H_\nu^1(M,g)}:= \int_M v^2 d \nu+\int_M |\nabla_g v|^2 {\rm dVol}_g.
\end{equation}

The validity of \eqref{ineq_poinc} has several interesting consequences, as discussed in \cite[Section 3]{GKL_measures}.
We collect them in the following result, see \cite[Section 3]{GKL_measures} for a proof and  \cite[Lemma 3.2]{BORS_abs_mul} for the last claim.
\begin{theorem}\label{theor_prel}
Suppose that \eqref{ineq_poinc} holds for some Radon measure  $\nu$ not supported on a single point. Then there exist $C_{1.\nu}>0, C_{2.\nu}>0$, depending only on $\nu$, such that 
\begin{equation}\label{ineq_equiv_norm}
C_{1.\nu}\norm{v}_{H_\nu^1(M,g)}\le \norm{v}_{H^1(M,g)} \le C_{2,\nu}\norm{v}_{H_\nu^1(M,g)}
\qquad\forall v\in H_\nu^1(M,g).
\end{equation}
In particular, as  $C^\infty(M)$ is dense in both spaces, $H_\nu^1(M,g)=H^1(M,g)$ with equivalent norms.
Furthermore, the embedding $H^1(M,g) \hookrightarrow L^2(M,\nu)$ is compact if and only if $H_\nu^1(M,g) \hookrightarrow L^2(M,\nu)$  is compact and  \eqref{ineq_poinc} holds.

Equivalently, for any $\delta$ there exists $C(\delta)>0$, depending only on $\delta$, such that 
\begin{equation}\label{ineq_compct}
\int_M|v|^2 d \nu \leq \delta\int_M |\nabla_g v|^2 {\rm dVol}_g + C(\delta)\int_M|v|^2{\rm dVol}_g
\qquad\forall v \in H^1(M,g).
\end{equation}
\end{theorem}

\begin{remark}\label{rem_equiv_norm}
Since $\mu(M)={\rm Vol}_g(M)=1$, the constant $C_{2.\nu}$ depends only on $K_\nu$ and $C_{1,\nu}$. More precisely, we can choose 
\begin{equation}
C_{2.\nu}:=2(1+K_\nu)(1+C_{1,\nu}),
\end{equation}
see  \cite[Proposition 3.12]{GKL_measures}.
\end{remark}

In view of the theorem above and classical spectral theory, the eigenvalues $\{\la_n(g,\nu)\}_{n \in \mathbb{N}}$ are actually the eigenvalues of the Laplace-Beltrami operator  with respect to the measure $\nu$, repeated according to their multiplicity, that is, for any $n \in \mathbb{N}$ there exists an eigenfunction 
$\varphi \in H^1(M,g)\setminus\{0\}$ which solves the equation
\begin{equation}\label{eq_eigen_func}
\int_M g( \nabla_g \varphi, \nabla_gv) \, {\rm dVol}_g=\la_n(g,\nu)  \int_M   \varphi v \, d \nu
\quad \text{ for any } v \in H^1(M,g).
\end{equation}
It is also  clear that
\begin{equation}
\la_0(g,\mu):=\inf_{E_0}\sup_{v \in E_0} \frac{\int_M |\nabla_g v|^2 \,  {\rm dVol}_g}{\int_M | v|^2 \, d \mu}=0,
\end{equation}
choosing as linear subspace of dimension $1$ the span of a constant function. 
Hence, we focus on  the case $n \ge 1$.

\subsection{Quantitative spectral stability for Radon measures with compact embeddings}
In this subsection, we recall a first order asymptotic expansion for eigenvalues associated to a family of compact Radon measures obtained in \cite[Section 3]{BORS_abs_mul}. We are then going to apply this result in the next section to outer variations of a Radon measure $\mu$.

Let us consider a family of positive Radon  measures $\{\mu_\e\}_{\e \in [0,1]}$ on $M$.
We suppose that  for any $\e \in [0,1]$ and  for any $\delta>0$ there exists $C(\delta)>0$, that depends only on $\delta$ such that
\begin{equation}\label{ineq_poinc_measures_mu}
\int_{M}|v|^2 d \mu_\varepsilon \le \delta\int_{M} |\nabla_g v|^2 {\rm dVol}_g + C(\delta)\int_{M}|v|^2 {\rm dVol}_g
\end{equation}
for any $v \in H^1(M,g)$, that is, we assume that \eqref{ineq_compct} holds uniformly with respect to $\e \in [0,1]$.
Furthermore, we suppose that  for any $v \in C^\infty(M)$ as $\e \to 0^+$
\begin{equation}\label{eq_conv_weak}
\int_{M} v \, d \mu_\e \to \int_{M} v \, d \mu.
\end{equation}
Conditions  \eqref{ineq_poinc_measures_mu} and \eqref{eq_conv_weak} can be characterized as in the next proposition and implies the validity of a uniform Poincaré inequality, that is, \eqref{ineq_poinc} holds uniformly with respect to $\e \in[0,1]$.
\begin{proposition}{\cite[Proposition 3.9]{BORS_abs_mul}}\label{prop_limit_mue}
A family of measures $\{\mu_\e\}_{\e \in [0,1]}$ satisfies \eqref{ineq_poinc_measures_mu} and \eqref{eq_conv_weak}
if and only if the following holds.
Suppose that $u_\e \rightharpoonup u$,  $v_\e \rightharpoonup v$ weakly in $H^1(M,g)$ as $\e \to 0^+$. Then  as $\e \to 0^+$
\begin{equation}\label{eq_limit_mue}
\int_M u_\e v_\e \, d\mu_\e\to  \int_M uv \, d\mu.
\end{equation}
\end{proposition}

\begin{proposition}{\cite[Proposition 3.6]{BORS_abs_mul}}\label{prop_poinc}
If a family of measures $\{\mu_\e\}_{\e \in [0,1]}$ satisfies \eqref{ineq_poinc_measures_mu} and \eqref{eq_conv_weak}
then there exists a constant $K>0$, which does not depend on $\e\in [0,1]$, such that 
\begin{equation}
\int_{M}|v-c_{v,\mu_\e}|^2 d \mu_\e \leq K \int_{M} |\nabla_g v|^2 {\rm dVol}_g
\end{equation}
for any $v \in H^1(M,g)$.
\end{proposition}
A a consequence, by Proposition \ref{prop_poinc} and   Remark \ref{rem_equiv_norm}, also \eqref{ineq_equiv_norm} holds uniformly with respect to $\e \in[0,1]$.

By Theorem \ref{theor_prel} and \eqref{ineq_poinc_measures_mu} for any $\e \in [0,1]$ the problem 
\begin{equation}\label{eq_eigen_outer_var}
-\Delta_g u+u\mu_\e=\la u \mu_\e\quad \text{ in }  M
\end{equation}
admits a sequence of positive and divergent eigenvalues $\{\la_n(g,\mu_\e)\}_{n \in \mathbb{N}}$, repeated according to their multiplicity which we denote simply with $\{\la_{\e, n}\}_{n \in \mathbb{N}}$ throughout this section.

An eigenfunction of \eqref{eq_eigen_outer_var} is then a non-trivial solution to the equation
\begin{equation}\label{eq_eigen_func_outer}
\int_M g( \nabla_g \varphi, \nabla_gv) \, {\rm dVol}_g+\int_M   \varphi v \, d \mu_\e=\la  \int_M   \varphi v \, d \mu_\e
\quad \text{ for any } v \in H^1(M,g)
\end{equation}
and for some $\la \ge 1$.

\begin{remark}\label{rem_equiv_prob}
It is of course equivalent to considering the problem \eqref{eq_eigen_outer_var} 
\begin{equation}\label{eq_eigen_outer_var_not_tran}
-\Delta_g u=\la u \mu_\e\quad \text{ in }  M
\end{equation}
since, denoting by $\{\tilde{\la}_n(\mu_\e,g)\}_{n \in \mathbb{N}}$ its spectrum, we have 
\begin{equation}
    \la_n(\mu_\e,g)=1+\tilde{\la}_n(\mu_\e,g), \text{ for any } n \in \mathbb{N}
\end{equation}
so that, in particular $\la_0(\mu_\e,g)=1$ and $\la_n(\mu_\e,g)- \la_n(\mu,g)=\tilde{\la}_n(\mu_\e,g)-\tilde{\la}_n(\mu,g)$.
The advantage of formulation \eqref{eq_eigen_func_outer} is the coercivity of the associated bilinear form.
\end{remark}

An interesting consequence of \eqref{eq_limit_mue} is spectral stability, as proved in \cite[Proposition 3.10]{BORS_abs_mul}.
\begin{proposition}\label{prop_stab_eigen_measures}
For any $n \in \mb{N}$
\begin{equation}\label{eq_stab_eigen_measures}
\lim_{\e \to 0^+} \la_{\e,n} =\la_{0,n}.
\end{equation}
\end{proposition}

To quantify  \eqref{eq_stab_eigen_measures} into a first order expansion, let us set some notation. Let 
\begin{align}
&E(\la_{0, n})\subset L^2(M,\mu_0) \text{  be  the eigenspace  associated to the eigenvalue } \la_{0,n}, \label{def_V}\\
&k_{0,n} \text{  be the dimension of }  E(\la_{0, n}), \label{def_k} \\
&\{\varphi_{0,n,i}\}_{i=1, \ldots, k_{0,n}} \text{ be an orthonormal basis of } E(\la_{0, n}) \text { of } L^2(M,\mu_0). \label{def_basis}
\end{align}
By standard variational methods, it is easy to see that 
for any $\varphi \in E(\la_{0, n})$,  the  ``linearized'' equation
\begin{equation}\label{eq_Ve_meausures}
\int_{M} g( \nabla u, \nabla v)  {\rm dVol}_g +  \int_M  u v \, d \mu_\e
= \la_{0,n}  \int_{M}   \varphi v \, d (\mu_0-\mu_\e) \quad \text{ for any } v \in H^1(M,g)
\end{equation} 
admits a unique solution $u=V_{\e,\varphi} \in H^1(M,g)$. Indeed, we may simply minimize the functional
\begin{equation}
J_{\e,\varphi}(v):=\frac{1}{2}\int_{M} g( \nabla v, \nabla v)  {\rm dVol}_g
-\la_{0,n}  \int_{M}   \varphi v \, d (\mu_0-\mu_\e)
\end{equation}
over $H^1(M,g)$.

Furthermore, let us define the bilinear form  in $L^2(M,\mu_0)$
\begin{equation}\label{def_qe}
h_\e:E(\la_{0, n})\times E(\la_{0, n}) \to \R, \qquad h_\e(\varphi,\psi):=\la_{0,n}\int_M V_{\e,\varphi} \psi \, d\mu_\e.
\end{equation}
Then we have the following result, see \cite[Theorem  3.12]{BORS_abs_mul}.

\begin{theorem}\label{theo_quantitative_abs}
Let $\{\mu_\e\}_{\e \in[0,1]}$ be Radon measures on a smooth manifold $(M,g)$. Assume that \eqref{ineq_poinc_measures_mu} and \eqref{eq_conv_weak} hold. 
Let $\la_{0, n}$ be an eigenvalue of \eqref{eq_eigen_outer_var} with $\e=0$ and such that $\la_{0,n-1}<\la_{0, n}<\la_{0,n+k_{0,n}}$ where $k_n$ denotes its multiplicity.

Then the bilinear form $h_\e$ defined in \eqref{def_qe} is symmetric and, denoting by $\{\beta_{\e,n,i}\}_{i=1,\ldots,k_{0,n}}$ its eigenvalues,
\begin{equation}
\la_{\e, n+i-1}-\la_{0, n+i-1}=
\la_{\e, n+i-1}-\la_{0, n}=\beta_{\e,n,i}+O(\delta_\e^2)+o(\tau_\e^2), \quad \text{ as } \e \to 0^+,
\end{equation}
for any $i=1,\ldots,k_{0,n}$, where 
\begin{align}
\delta_\e:=\sup\{\norm{V_{\e,\varphi}}_{L^2(M,\mu_\e)}:\ \varphi \in E(\la_{0, n}), \,\,\norm{\varphi}_{L^2(M,\mu)}=1\},\\
\tau_\e^2:=\sup\left\{\lambda_{0,N}\left|\int_{M}\varphi V_{\e,\varphi}\, d\mu_\e\right|:\ \varphi \in E(\la_{0, n}), \,\,\norm{\varphi}_{L^2(M,\mu)}=1\right\}.
\end{align}
\end{theorem}

\subsection{Outer Variations}\label{sec_outer_var}
Let the Radon measures $\mu$, $\nu$ be such that the embeddings $H^1(M,g)\hookrightarrow L^2(M,\mu)$  and $H^1(M,g)\hookrightarrow L^2(M,\nu)$  are compact.

For any $\e \in [0,1]$ let us consider the outer variation of the measure $\mu$ with respect to  $\nu$ given by \eqref{def_kue}.
It is easy to check that, by definition of $\mu_\nu(\e)$ the constants $K_{\mu_\nu(\e)}$ in \eqref{ineq_poinc} can be chosen independently on $\e$, so that it depends  only on $\mu,\,\nu$. Similarly  $\mu_\nu(\e)$ satisfies  \eqref{ineq_compct} uniformly on $\e$. Furthermore, by definition of $\mu_\nu(\e)$, for any $u \in C^\infty(M)$
\begin{equation}
\int_M u \, d \mu_\nu(\e) \to \int_M ud \mu \quad \text{ as } \e \to 0^+.
\end{equation} 
Hence, the assumptions of Theorem \ref{theo_quantitative_abs} are satisfied. 

Finally, in this case, \eqref{eq_Ve_meausures} can be written as
\begin{equation}\label{eq_Ve_meausures_outer}
\int_M g( \nabla_g u, \nabla_gv) \, {\rm dVol}_g +  \int_M  u v \, d \mu_\nu(\e)=\e(\la_n(\mu,g)+1)  \int_M   \varphi v \, d (\mu-\nu)
\end{equation}
for any  $v \in H^1(M,g)$. Let us also define the bilinear form on $E(\la_n(\mu,g))\subset L^2(M,\mu)$
\begin{equation}\label{def_q}
h:E(\la_n(\mu,g))\times E(\la_n(\mu,g)) \to \R, \qquad h(\varphi,\psi):=(\la_n(\mu,g)+1)  \int_M \varphi \psi \, d(\mu-\nu).
\end{equation}
Then we have the following result.

\begin{proposition}\label{prop_eigen_outer_var}
If $h$ is the bilinear form $h$ defined in \eqref{def_q}, one has
\begin{equation}\label{eq_qe_q}
\lim_{\e \to 0^+}\frac{1}{\e} h_\e(\varphi,\psi)=h(\varphi,\psi)\qquad\forall \varphi,\psi \in E(\la_n(\mu,g)).
\end{equation}
Then, if $\la_{n-1}(\mu,g)<\la_n(\mu,g)<\la_{n+k_n}(\mu,g)$, where $k_n$ is the multiplicity of $\la_n(\mu,g)$, 
\begin{equation}\label{eq_eigen_outer_var_abst}
\la_{n+i-1}(\mu_\nu(\e),g)-\la_n(\mu,g)=\beta_{n,i}\e+o(\e), \text{ as } \e \to 0^+
\end{equation}
 for any $i\in\{1,\dots,k_n\}$, where $\{\beta_{n,i}\}_{i=1,\ldots,k_n}$ are the  eigenvalues of $h$ on $E(\la_n(\mu,g))\subset L^2(M,\mu)$.
\end{proposition}

\begin{proof}
In order to prove \eqref{eq_qe_q} is enough to check that 
\begin{equation}
\lim_{\e \to 0^+}\frac{1}{\e} h_\e(\varphi,\varphi)=h(\varphi,\varphi)\qquad\forall \varphi \in E(\la_{0, n}).
\end{equation}
Let us define for any $\e \in (0,1]$ the function 
\begin{equation}
\widetilde{V}_{\e,\varphi}:=\frac{V_{\e,\varphi}}{\e} \in H^1(M,g).
\end{equation}
Testing \eqref{eq_Ve_meausures_outer} with $v=\widetilde{V}_{\e,\varphi}$ we obtain
\begin{equation}
\int_Mg(\nabla_g \widetilde{V}_{\e,\varphi}, \nabla_g \widetilde{V}_{\e,\varphi}) \, {\rm dVol}_g+\int_M |\widetilde{V}_{\e,\varphi}|^2d\mu_\nu(\e)
=(\la_n(\mu,g)+1) \int_M \widetilde{V}_{\e,\varphi}\varphi  \, d(\mu-\nu).
\end{equation}
Since  the constants in \eqref{ineq_equiv_norm} and \eqref{ineq_compct} are uniform in $\e$ as previously observed, it follows that
\begin{multline}
\norm{\widetilde{V}_{\e,\varphi}}_{H^1(M,g)}^2 \le C_{2,\mu_\nu(\e)}\norm{\widetilde{V}_{\e,\varphi}}_{H_{\mu_\nu(\e)}^1(M,g)}^2 \\
\le \la_n(\mu,g)C_{2,\mu_\nu(\e)}(1+C(1))\norm{\widetilde{V}_{\e,\varphi}}_{H^1(M,g)}\norm{\varphi}_{H^1(M,g)}
\end{multline}
thus $\{\widetilde{V}_{\e,\varphi}\}_{\e \in (0,1]}$ is bounded in $H^1(M,g)$.
Up to a subsequence,
$\widetilde{V}_{\e,\varphi} \rightharpoonup \tilde{V}_\varphi$, with $\tilde{V}_\varphi \in H^1(M,g)$ solving the equation
\begin{equation}\label{eq_Vvarphi}
\int_M g(\nabla_g \tilde{V}_\varphi,\nabla_g v) \, {\rm dVol}_g+\int_M \tilde{V}_{\varphi} v d\mu
= (\la_n(\mu,g)+1) \int_M   \varphi v \, d (\mu-\nu)
\end{equation}
for any $v \in H^1(M,g)$.
Hence, by \eqref{eq_limit_mue}, as $\e \to 0^+$
\begin{multline}
\frac{1}{\e}h_\e(\varphi,\varphi)=\frac{1}{\e}(\la_n(\mu,g)+1) \int_M V_{\e,\varphi} \varphi \, d\mu_\nu(\e)=(\la_n(\mu,g)+1)\int_M \widetilde{V}_{\e,\varphi} \varphi \, d\mu_\nu(\e)\\
\to (\la_n(\mu,g)+1) \int_M \tilde{V}_\varphi \varphi \, d\mu=\int_M g( \nabla_g \tilde{V}_\varphi, \nabla_g\varphi) \, {\rm dVol}_g
+\int_M \tilde{V}_\varphi \varphi \,  d \mu\\
=(\la_n(\mu,g)+1)   \int_M   |\varphi|^2 \, d (\mu-\nu)
=h(\varphi,\varphi),
\end{multline}
in view of  \eqref{eq_eigen_func_outer} tested with $v=\tilde{V}_\varphi$ and \eqref{eq_Vvarphi} tested with $v=\varphi$.   Taking into account Theorem \ref{theo_quantitative_abs}, the first order expansion \eqref{eq_eigen_outer_var_abst} follows from \eqref{eq_qe_q}.
\end{proof}

\subsection{Regularization of Radon measures with compact embedding.} Let   $\mu$ be a Radon measure such that the embedding $H^1(M,g)\hookrightarrow L^2(M,\mu)$ is  compact. In order to prove Proposition \ref{prop_reg}, in this section we provide a natural one parameter regularization $\mu_\e$ of $\mu$ such that $\mu_\e$ is absolutely continuous with respect to ${\rm Vol}_g$ with a smooth, positive density.

Let $P_\epsilon\mu$ be the heat semigroup associated with the Laplace-Beltrami operator $\Delta_g$. 
We consider the dual semigroup $P^*_\e$ on Radon probability measures on $M$, defined by
\begin{eqnarray}
\int_M \psi  \, d P^*_\e\mu :=\int_M P_\e\psi  \, d \mu\qquad\forall\psi\in C^0(M).
\end{eqnarray}
By Riesz's Theorem and the Feller property of $P_\e$, and since $P_\epsilon 1=1$, $P^*_\e\mu$ is well defined.  Furthermore, it is well known (see for instance \cite{BGL_book}) that $P_\e^*\mu$ is absolutely continuous with respect to ${\rm Vol}_g$, with a smooth and positive density and that
$P^*_\e\mu$ weakly$^*$ converges to $\mu$ as $\e\to 0^+$.

\begin{proposition}\label{prop_reg_uniform_compct_emb}
For any $\delta>0$ there exists a constant $C(\delta)>0$, that does not depend on $\e \in (0,1]$, such that 
\begin{equation}\label{ineq_reg_uniform_compct_emb}
\int_M v^2 d P^*_\e\mu \leq \delta\int_M |\nabla_g v|^2 {\rm dVol}_g + C(\delta)\int_M|v|^2{\rm dVol}_g
\qquad\forall v \in H^1(M,g).
\end{equation}
\end{proposition}

\begin{proof}
Letting $w_\e:=\sqrt{P_\e(v^2)} \in C^\infty(M)$, by Theorem  \ref{theor_prel} for any $\delta>0$ there exists a constant $C(\delta)>0$ depending only on $\mu$ and $\delta$ such that 
\begin{equation}
\int_M P_\e(v^2) d \mu=\int_M w_\e^2 d \mu  \leq \delta\int_M |\nabla_g w_\e|^2 {\rm dVol}_g + C(\delta)\int_Mw_\e^2{\rm dVol}_g.
\end{equation}
For any $\e >0$, since $P_\e$ is self-adjoint with respect to the scalar product $L^2(M,g)$ and $P_\e1=1$,
\begin{equation}
\int_Mw_\e^2{\rm dVol}_g=\int_M P_\e(v^2) {\rm dVol}_g=\int_M v^2 {\rm dVol}_g.
\end{equation}
Furthermore, by the classical Bakry-Emery theory (see for instance \cite{B_cal}), for any $f>0$ one has
\begin{equation}
\frac{|\nabla_g P_\e f|^2}{P_\e f} \le e^{-2K\e} P_\e\left(\frac{|\nabla_g f|^2}{f}\right),
\end{equation}
where $K$ is the largest constant that satisfies ${\rm Ric}_M\geq Kg$. Hence,
\begin{equation}
 |\nabla_g w_\e|^2= \left|\nabla_g \sqrt{P_\e(v^2)}\right|^2=\frac{|\nabla_g P_\e (v^2)|^2}{4P_\e(v^2)} 
 \le e^{-2K\e} P_\e(|\nabla_g v|^2).
\end{equation}
It follows that
\begin{equation}
  \int_M |\nabla_g w_\e|^2 {\rm dVol}_g \le   e^{-2K\e}  \int_M P_\e(|\nabla_g v|^2){\rm dVol}_g =\int_M |\nabla_g v|^2{\rm dVol}_g,
\end{equation}
and so we have proved \eqref{ineq_reg_uniform_compct_emb}.
\end{proof}

By Proposition \ref{prop_reg_uniform_compct_emb} and Theorem \ref{theor_prel}, the embedding $H_\nu^1(M,g) \hookrightarrow L^2(M,\nu_\e)$ is compact.
It follows that  for any $\e \in [0,1]$ the problems 
\begin{equation}
-\Delta_g u=\la u \mu_\e\quad \text{ in }  M
\end{equation}
admit a sequence of positive and divergent eigenvalues $\{\la_n(g,\mu_\e)\}_{n \in \mathbb{N}}$, repeated according to their multiplicity.

In view of  Proposition \ref{prop_reg_uniform_compct_emb}, since $P^*_\e\mu$ weakly$^*$ converge to $\mu$ as $\e\to 0^+$, from Proposition \ref{prop_stab_eigen_measures} we can deduce 
\begin{equation}
\lim_{\e \to 0^+} \la_n(g,\mu_\e) =\la_n(g,\mu).
\end{equation}
The proof of Proposition \ref{prop_reg} is now simple.
\begin{proof}[Proof of Proposition \ref{prop_reg}.]
Clearly by \eqref{prob_max_mes_comp} and \eqref{prob_max_mes_reg}
\begin{equation}
\bar{\Lambda}_n(M,g)\le \widetilde{\Lambda}_n(M,g).
\end{equation}
On the other hand, let $\delta>0$ and $\mu$  be such that
\begin{equation}
\widetilde{\Lambda}_n(M,g) \le \lambda_n(\mu,g)+\delta.
\end{equation}
The family of measures  $\{\mu_\e\}_{\e \in (0,1]}$  satisfies  
\begin{equation}
\lim_{\e \to 0^+} \la_n(\mu_\e,g) =\la_n(\mu,g).
\end{equation}
Hence, for $\e>0$ small enough depending on $\delta$, since $\mu_\e$ is absolutely continuous with respect to ${\rm Vol}_g$ with smooth and positive density, 
\begin{equation}
\widetilde{\Lambda}_n(M,g)\le \lambda_n(\mu_\e,g)+2\delta \le \bar{\Lambda}_n(M,g)+2\delta.
\end{equation}
Since $\delta>0$ is arbitrary
\begin{equation}
\widetilde{\Lambda}_n(M,g)\le \bar{\Lambda}_n(M,g)
\end{equation}
and we have completed the proof.
\end{proof}

\section{Regularity theory for stationary measures}\label{sec_reg}

Let $\mu$ be as in Theorem \ref{theor_main}, that is \eqref{hp_compct} holds and $\mu$ is stationary for  $n \in \mathbb{N}\setminus\{0\}$
and let $E(\lambda_n(g,\mu))\subset L^2(M,\mu)$ is the eigenspace associated to $\lambda_n(g,\mu)$.  Let us also define 
\begin{equation}\label{A_n}
A_n:=\left\{u \in E(\lambda_n(g,\mu)): \int_M u^2 \, d\mu=1\right\}.
\end{equation}
In this section we are going to use the first order expansion for outer variations of $\mu$ to prove Theorem \ref{theor_main}.
They key result is the following inequality which follows from \eqref{eq_eigen_outer_var_abst} and the stationarity of $\mu$.

\begin{proposition}\label{prop_ineq_mu}
Let $\mu$ be such that  \eqref{hp_compct} holds and $\mu$ is stationary for some $n \in \mathbb{N}\setminus\{0\}$. Then for any probability Radon measure $\nu$ satisfying \eqref{hp_compct}
\begin{equation}\label{ineq_mu}
\sup_{\varphi \in A_n} \int_M |\varphi|^2 d \nu \ge 1.
\end{equation}
\end{proposition}

\begin{proof}
Denoting by $k_n$ the multiplicity of $\la_n(\mu,g)$, there exists $\ell_n\in \{1,\ldots,k_n\}$  such that 
\begin{equation}
\la_{n-\ell_n}(\mu,g)<\la_n(\mu,g) \quad  \text{ and } \quad \la_{n-i+1}(\mu,g)=\la_n(\mu,g),
\end{equation}
for any $i=1,\ldots,\ell_n$. By Proposition \ref{prop_eigen_outer_var} applied to $\mu$ and the eigenvalue $\la_{n-\ell_n+1}(\mu,g)$,
since $\mu$  is stationary, $\beta_{n,\ell_n}\le 0$. In particular, we surely have  $\beta_{n,1}\le \beta_{n,\ell_n}\le 0$.
By definition of a Rayleigh quotient 
\begin{equation}
0\ge \beta_{n,1}=(\la_n+1)\min_{u \in A_n}\int_M \varphi^2 \, d (\mu-\nu),
\end{equation}
thus we have proved \eqref{ineq_mu}.
\end{proof}

Let $E(\lambda_n(g,\mu)$ and  $k_n$ be as in \eqref{def_V} and  \eqref{def_k} for the measure $\mu$. Taking some ideas from \cite[Lemma 4.6]{K_var}, we can deduce the following proposition from \eqref{ineq_mu}.

\begin{proposition}\label{prop_sum_varphi}
If $\mu$  is stationary for some $n \in \mathbb{N}\setminus \{0\}$ then 
there exist $m_n \in \{1,\dots , k_n\}$, positive numbers $\{\gamma_1, \dots. \gamma_{m_n}\}$ and an orthonormal basis $\{\varphi_{n,i}\}_{i=1,\dots, k_n}$ of the eigenspace 
$E(\lambda_n(g,\mu))$ in $L^2(M,\mu)$ such that 
\begin{equation}\label{eq_sum_varphi_M_n}
\sum_{i=1}^{m_n}\gamma_i\varphi_{n,i}^2=1 \quad \mu-\text{a.e. in } M,
\end{equation}
and 
\begin{equation}\label{ineq_varphi_n}
\sum_{i=1}^{m_n}\gamma_i\varphi_{n,i}^2\ge 1 \quad {\rm Vol}_g-\text{a.e. in } M.
\end{equation}
\end{proposition}

\begin{proof} 
Let us define $\sigma:=\mu +{\rm Vol}_g$ and consider the measure $\nu_h:= h\sigma$ where $h \in L^\infty(M,\sigma)$, 
$h(x) \ge 0$ for $ \sigma$-a.e. $x \in M$, and $\int_M h \, d\sigma=1$. We may test \eqref{ineq_mu} thus obtaining 
\begin{equation}\label{proof_prop_sum_varphi_1}
\sup_{\varphi \in A_n} \int_M |\varphi|^2 h  \, d\sigma \ge 1.
\end{equation}
Let us define the compact and convex set 
\begin{equation}
K:={\rm{conv}}\left\{\varphi^2:\ \varphi \in A_n\right\} \subset  L^1(M,\sigma).
\end{equation}
We claim that there exists $\rho \in K$ such that for  $\rho \ge 1$ for $\sigma$-a.e. $x \in M$. We argue by contradiction supposing that
\begin{equation}
K\cap C= \emptyset, \quad \text{ where } C:=\left\{f \in L^1(\sigma): f(x)\ge 1 \text{ for $\sigma$-a.e. } x \in M \right\}.
\end{equation}
By the  Hahn-Banach separation Theorem, we may find $\eta \in L^\infty(M,\sigma)$ such that 
\begin{equation}
\sup_{\rho \in K}\int_M \rho  \eta \, d\sigma < \inf_{f \in C}\int_M \eta f \, d\sigma.
\end{equation}
We must have $\eta(x) \ge 0$ for $\sigma$-a.e. $x \in M$, otherwise the right hand-side is $-\infty$ while the left hand-side is finite. Furthermore, 
$\eta > 0$  in a set of positive $\sigma$-measure otherwise the inequality cannot be strict. As a consequence
\begin{equation}
\inf_{f \in C}\int_M \eta f \, d\sigma=\int_M \eta \, d\sigma
\end{equation}
thus defining 
\begin{equation}
h:=\frac{\eta}{\int_M \eta \, d\sigma}
\end{equation}
we have $h \in L^\infty(M,\sigma)$, $h \ge 0$ $\sigma$-a.e. in $M$ and $\int_M h \, d\sigma=1$.
Hence,
\begin{equation}
\sup_{\varphi \in A_n} \int_M |\varphi|^2 h  \, d\sigma \le \sup_{\rho \in K}\int_M \rho  h \, d\sigma < 1,
\end{equation}
thus  we have reached a contradiction with \eqref{proof_prop_sum_varphi_1}.
It follows that there exist $N \in \mathbb{N}\setminus \{0\}$ and $t_j \in (0,1]$, with $\sum_{j=1}^Nt_j=1$, $\psi_j \in A_n$ for $j=1, \ldots, N$ such that
\begin{equation}\label{proof_prop_sum_varphi_2}
\rho(x):=\sum_{j=1}^N t_j \psi_j^2(x) \ge 1  \text{ for $\sigma$-a.e. } x \in M.
\end{equation}
Since $\psi_j \in A_n$ for any $j=1,\ldots, N$, we obtain
\begin{equation}
\int_M\rho \, d\mu=\sum_{j=1}^Nt_j =1,
\end{equation}
so that we must have, by \eqref{proof_prop_sum_varphi_2} and definition of $\sigma$ as $\mu +{\rm Vol}_g$,
\begin{equation}\label{proof_prop_sum_varphi_3}
\rho(x)=1  \text{ for $\mu$-a.e. } x \in M \quad  \text{ and } \rho(x) \ge 1 \quad \text{ for ${\rm Vol}_g$-a.e. } x \in M.
\end{equation}
Let $\{\widetilde\varphi_{n,i}\}_{i=1,\ldots, k_n}$ be an orthonormal basis of the eigenspace 
$E(\lambda_n(g,\mu))$ in $L^2(M,\mu)$.
We may write for some $\alpha_{i,j} \in \R$ with $\sum_{i=1}^{k_m}\alpha^2_{i,j}=1$
\begin{equation}
\psi_j=\sum_{i=1}^{k_n}\alpha_{i,j}\widetilde\varphi_{n,i}
\end{equation}
so that 
\begin{equation}
\rho=\sum_{j=1}^Nt_j\sum_{i,i'=1}^{k_n} \alpha_{i,j}\alpha_{i',j}\widetilde\varphi_{n,i} \widetilde\varphi_{n,i'}
=\sum_{i,i'=1}^{k_n}\left(\sum_{j=1}^Nt_j \alpha_{i,j}\alpha_{i',j}\right)\widetilde\varphi_{n,i} \widetilde\varphi_{n,i'}=
B \widetilde\varphi_n \cdot \widetilde\varphi_n,
\end{equation}
where $\widetilde \varphi_n=(\widetilde\varphi_{n,i})_{i=1,\dots, k_n} $ and the $k_n \times k_n$ matrix $B$ is defined as 
\begin{eqnarray}
B:=\left(\sum_{j=1}^Nt_j \alpha_{i,j}\alpha_{i',j}\right)_{i,i'=1, \dots, k_n}.
\end{eqnarray}
Clearly $B$ is symmetric and it is also semidefinite, since for any $\xi \in \R^{k_n}$ one has
\begin{equation}
B \xi \cdot \xi=\sum_{j=1}^Nt_j\sum_{i,i'=1}^{k_n} \alpha_{i,j}\alpha_{i',j} \xi_i \xi_{i'}
=\sum_{j=1}^Nt_j\left(\sum_{i=1}^{k_n} \alpha_{i,j} \xi_i \right)^2\ge 0.
\end{equation}
Denoting with $\gamma_1, \ldots, \gamma_{m_n}$ the  positive eigenvalues of $B$ with $m_n \le k_n$,
we have  $B=QDQ$ for some orthogonal matrix $Q$, where $D$ is a diagonal matrix having $\gamma_1, \ldots, \gamma_{m_n}$ as first $m_n$ entries on its diagonal.
We conclude that 
\begin{equation}\label{proof_prop_sum_varphi_4}
\rho=B \widetilde\varphi_n \cdot \widetilde\varphi_n=QDQ\widetilde\varphi_n \cdot \widetilde\varphi_n
=DQ \widetilde \varphi_n \cdot Q \widetilde \varphi_n=\sum_{i=1}^{m_n}\gamma_i\varphi_{n,i}^2,
\end{equation}
where, letting $Q=(q_{i,j})_{i,j,=1\dots m_n }$, we have set 
\begin{eqnarray}
\varphi_{n,i}:=\sum_{j=1}^{k_m}q_{i,j}\widetilde\varphi_{n,j}
\qquad\text{for any $i=1, \dots,k_n$.}
\end{eqnarray}
It is easy to see that $\{\varphi_{n,i}\}_{i=1,\dots,k_n}$ is an orthonormal basis of eigenfunctions  of $E(\lambda_n(g,\mu))$ in $L^2(M,\mu)$. Indeed,
\begin{equation}
\int_M \varphi_{n,i} \varphi_{n,i'} \, d\mu
=\int_M \left(\sum_{j=1}^{k_m}q_{i,j}\widetilde\varphi_{n,j}\right)\left(\sum_{j=1}^{k_m}q_{i',j}\widetilde\varphi_{n,j} \right) \, d\mu
=\sum_{j=1}^{k_m}q_{i,j}q_{i',j}=\delta_{i,i'},
\end{equation}
where $\delta_{i,i'}=0$ if $i \ne i'$ and  $\delta_{i,i'}=1$ if $i=i'$, since $Q$ is an orthogonal matrix.

Hence, by \eqref{proof_prop_sum_varphi_3} and \eqref{proof_prop_sum_varphi_4} we have completed the proof.
\end{proof}

Let $m_n$, $\{\gamma_1, \dots,\gamma_{m_n}\}$ and $\{\varphi_{n,i}\}_{i=1,\dots,k_n}$ be as in Proposition \ref{prop_sum_varphi}. We define 
\begin{equation}
\Phi_n:M \to \R^{m_n}, \quad \Phi_n:=(\sqrt{\gamma_1}\varphi_{n,1},\dots, \sqrt{\gamma_{m_n}}\varphi_{n,m_n})
\end{equation}
and we are going to use the notation
\begin{equation}
|\Phi_n|^2=\sum_{i=1}^{m_n}\gamma_i|\varphi_{n,i}|^2 \quad \text{ and } \quad |\nabla \Phi_n|^2=\sum_{i=1}^{m_n}\gamma_i|\nabla \varphi_{n,i}|^2.
\end{equation}

\begin{proposition}\label{prop_eq_mu}
For any $\psi \in H^1(M,g) \cap L^\infty(M, {\rm dVol}_g)$ one has
\begin{equation}\label{eq_mu}
\la_n(g,\mu)\int_M  \psi \, d \mu=\int_M |\nabla_g \Phi_n|^2 \psi+\frac{1}{2}g(\nabla_g |\Phi_n|^2, \nabla_g \psi)  \, {\rm dVol}_g.
\end{equation}
\end{proposition}
\begin{proof}
Let $\psi \in H^1(M,g) \cap L^\infty(M, {\rm dVol}_g)$. Then $\psi \varphi_{n,i}\in H^1(M,g)$ and so testing \eqref{eq_eigen_func} 
for $\varphi_{n,i}$ with $v=\psi \varphi_{n,i}$  we obtain
\begin{equation}
\int_M |\nabla_g \varphi_{n,i}|^2 \psi+g(\nabla_g \varphi_{n,i}, \nabla_g \psi) \varphi_{n,i} \, {\rm dVol}_g
=\la_n(g,\mu)  \int_M   |\varphi_{n,i}|^2 \psi \, d \mu.
\end{equation}
for any $i=1, \dots, m_n$. Summing over $i \in \{1, \dots, m_n\}$, by \eqref{eq_sum_varphi_M_n} 
\begin{equation}
\int_M |\nabla_g \Phi_n|^2 \psi+\frac{1}{2}g(\nabla_g |\Phi_n|^2, \nabla_g \psi)  \, {\rm dVol}_g
= \la_n(g,\mu)  \int_M  \psi \, d \mu,
\end{equation}
that is, we have proved \eqref{eq_mu}.
\end{proof}

\begin{proposition}\label{prop_varphi_bouned}
For any $i=1, \ldots, m_n$ one has $|\varphi_{n,i}| \le \norm{\varphi_{n,i}}_{L^\infty(M,\mu)}$ for ${\rm Vol}_g$ a.e. $x$ in $M$.
In particular,  $|\Phi_n|^2 \in H^1(M,g)$  and for any $\psi \in H^1(M,g)\cap L^\infty(M, {\rm dVol}_g)$
\begin{equation}\label{eq_varphin2}
\int_M(|\Phi_n|^2-1)g(\nabla_g |\Phi_n|^2, \nabla_g \psi)\, {\rm dVol}_g
=-\int_M (|\nabla_g |\Phi_n|^2|^2 +2( |\Phi_n|^2-1)|\nabla_g \Phi_n|^2)\psi\, {\rm dVol}_g,
\end{equation}
that is,  $|\Phi_n|^2$ is a weak solution of the equation
\begin{equation}\label{eq_varphin2_strong_bis}
{\rm div}\bigl((|\Phi_n|^2-1)\nabla_g |\Phi_n|^2\bigr)=
|\nabla_g |\Phi_n|^2|^2 +2( |\Phi_n|^2-1)|\nabla_g \Phi_n|^2.
\end{equation}
\end{proposition}
\begin{proof}
Let us define for any $i=1, \ldots, m_n$
\begin{equation}
c_i:=\norm{\varphi_{n,i}}_{L^\infty(M,\mu)} \quad  \text{ and } \quad f_i:=\max\{\min\{\varphi_{n,i}, c_i\},-c_i\}.
\end{equation}
Testing \eqref{eq_eigen_func} with $f_i$ we obtain
\begin{equation}
\int_M |\nabla_g f_i|^2 \, {\rm dVol}_g=\la_n(g,\mu)  \int_M   f_i \varphi_{n,i}  \, d \mu.
\end{equation}
Since $|\varphi_{n,i}| \le c_i$ for $\mu$-a.e. $x \in M$, it follows that $f_i=\varphi_{n,i} $ for $\mu$ a.e. $x \in M$ thus 
\begin{equation}
\int_M |\nabla_g f_i|^2 \, {\rm dVol}_g=\la_n(g,\mu)  \int_M   |\varphi_{n,i}|^2 \, d \mu= \int_M   |\nabla_g \varphi_{n,i}|^2 \, {\rm dVol}_g.
\end{equation}
In particular, 
\begin{equation}
\int_{\{|\varphi_{n,i}|\ge c_i\}}   |\nabla_g \varphi_{n,i}|^2 \, {\rm dVol}_g=0
\end{equation}
so that $|\varphi_{n,i}| \le c_i$ a.e. in $M$.

Testing \eqref{eq_mu} with $|\Phi_n|^2 \psi$, by \eqref{eq_eigen_func}, we obtain 
\begin{multline}
\sum_{i=1}^{m_n}\int_Mg(\nabla_g \varphi_{n,i}, \nabla_g (\varphi_{n,i}\psi))  {\rm dVol}_g
=\la_n(g,\mu)\int_M  |\Phi_n|^2 \psi \, d \mu\\
=\int_M |\nabla_g \Phi_n|^2 |\Phi_n|^2\psi+\frac{1}{2}|\nabla_g |\Phi_n|^2|^2\psi + \frac{1}{2}|\Phi_n|^2g(|\nabla_g |\Phi_n|^2, \nabla_g \psi)  \, {\rm dVol}_g
\end{multline}
thus 
\begin{multline}
\int_M|\nabla_g \Phi_n|^2 \psi +\frac{1}{2}g(\nabla_g |\Phi_n|^2,\nabla_g \psi) \, {\rm dVol}_g\\
=\int_M |\nabla_g \Phi_n|^2 |\Phi_n|^2\psi+\frac{1}{2}|\nabla_g |\Phi_n|^2|^2\psi 
+ \frac{1}{2}|\Phi_n|^2g(|\nabla_g |\Phi_n|^2, \nabla_g \psi)  \, {\rm dVol}_g,
\end{multline}
that is, we have proved \eqref{eq_varphin2}. 

%A formal integration by part on the left hand side of \eqref{eq_varphin2} shows that \eqref{eq_varphin2} is a weak formulation for \eqref{eq_varphin2_strong}.
\end{proof}

The PDE \eqref{eq_varphin2_strong_bis} in divergence form is degenerate, hence, without appealing to the regularity theory for this class of equations, it is not obvious to give a meaning to $\Delta_g|\Phi_n|^2$. However, a formal manipulation gives the more handy expression
\begin{equation}\label{eq_varphin2_strong}
(|\Phi_n|^2-1)(-\Delta_g|\Phi_n|^2+2|\nabla_g \Phi_n|^2)=0,
\end{equation}
suggesting that $|\Phi_n|^2$ is subharmonic in the set $\{|\Phi_n|^2>1\}$. 
Therefore, it is reasonable to expect that $|\Phi_n|^2= 1$  q.e. in $M$.  We give a simple proof, based on \eqref{eq_varphin2}, in  the next proposition. 

\begin{proposition}
We have that 
\begin{equation}\label{eq_sum_varphi_M_n_vol}
|\Phi_n|^2= 1 \quad \text{q.e. in } M.
\end{equation}
\end{proposition}
\begin{proof}
Letting $\rho:=|\Phi_n|^2$,
since $\rho \ge 1$ by \eqref{ineq_varphi_n}, for any $\psi \in H^1(M,g)\cap L^\infty(M)$, $\psi \ge 0$
\begin{equation}
\int_M(\rho-1)g(\nabla_g \rho, \nabla_g \psi)\, {\rm dVol}_g \le 0,
\end{equation}
thanks to \eqref{eq_varphin2}.
Equivalently, letting $w:=\rho-1$,  for any $\psi \in H^1(M,g)\cap L^\infty(M)$, $\psi \ge 0$
\begin{equation}
\int_Mg(\nabla_g w^2, \nabla_g \psi)\, {\rm dVol}_g \le 0.
\end{equation}
Testing with $w^2$ we conclude that $w^2$ is constant in $M$ and so $w^2=0$ by \eqref{eq_sum_varphi_M_n}. It follows that $\rho=1$ q.e. in $M$ and so we have proved \eqref{eq_sum_varphi_M_n}.
\end{proof}

\begin{proof}[Proof of Theorem \ref{theor_main}.]
If  $m_n=1$ then \eqref{eq_sum_varphi_M_n_vol} implies that   $|\varphi_{n,1}|^2=\gamma_1^{-1}$. Since $\varphi_i \in H^1(M,g)$ and $M$ is connected either $\varphi_{n,1}=1$ or $\varphi_{n,1}=-1$  on $M$ which is a contradiction with $\la_n(g,\mu)>0$.
Hence, $k_n \ge m_n \ge  2$. 

For any $i=1, \dots, m_n$, testing \eqref{eq_mu} with $\varphi_{n,i}\psi$ by  \eqref{eq_eigen_func} we obtain
\begin{equation}
\int_M g(\nabla_g \varphi_{n,i}, \nabla_g \psi) \,  {\rm dVol}_g= \int_M |\nabla_g \Phi_n|^2 \varphi_{n,i}\psi  \, {\rm dVol}_g.
\end{equation}
Hence, $\Phi_n:M \to \mathbb{S}^{m_n-1}$ is an harmonic map. By classical regularity theory for harmonic maps that take value on spheres, it is smooth if $d=2$. We refer, for example, to the classical book \cite{H_book}.
By \eqref{eq_mu} and \eqref{eq_sum_varphi_M_n_vol}, 
it follows that for any $\psi \in L^\infty(M)\cap H^1(M,g)$
\begin{equation}
\la_n(g,\mu)\int_M  \psi \, d \mu=\int_M |\nabla_g \Phi_n|^2 \psi  \, {\rm dVol}_g.
\end{equation}
Hence, $\mu$ is absolutely continuous with respect to $ {\rm Vol}_g$ and 
\begin{equation}
\mu=\frac{|\nabla_g \Phi_n|^2}{\la_n(g,\mu)}{\rm Vol}_g.
\end{equation}
This completes the proof, since, if $d=2$, $|\nabla_g \Phi_n|^2=0$ in at most a finite number of points, see \cite{H_book}.
\end{proof}

\subsection*{Acknowledgements}
\noindent
 G.~Siclari is partially supported by the 2026
INdAM--GNAMPA project ``Asymptotic analysis of variational problems''
(CUP E53C25002010001).

\bibliographystyle{acm}
\bibliography{references}	
\end{document}